\documentclass[number]{elsarticle}

\usepackage{amsmath}
\usepackage{amsthm}
\usepackage{amssymb}
\usepackage{amsfonts}
\usepackage{graphicx}
\newtheorem{thm}{Theorem}[section]
\newtheorem*{thm*}{Theorem}
\newtheorem{lem}[thm]{Lemma}
\newtheorem{cor}[thm]{Corollary}

\theoremstyle{definition}

\DeclareMathOperator{\ord}{ord}
\DeclareMathOperator{\cha}{char}

\def\notdivides{\mathrel{\kern-3pt\not\!\kern3pt\bigm|}}
\def\divides{\mathrel{\negmedspace|\negmedspace}}

\def\ge{\geqslant}

\begin{document}

\begin{frontmatter}

\title{Polynomial Zsigmondy theorems}

\author{Anthony Flatters}
\author{Thomas Ward}
\ead{t.ward@uea.ac.uk}
\address{School of Mathematics, University of East Anglia,
Norwich NR4 7TJ, UK}

\begin{abstract}
We find analogues of the
primitive divisor results of Zsigmondy, Bang,
Bilu--Hanrot--Voutier, and Carmichael
in polynomial rings, following
the methods of Carmichael.
\end{abstract}

\begin{keyword}
Zsigmondy theorem \sep Polynomial ring \sep Primitive divisor
\MSC[2010] 11A41 \sep 11B39
\end{keyword}

\end{frontmatter}

A prime divisor of a term~$a_n$ of a sequence~$(a_n)_{n\ge1}$
is called primitive if it divides no earlier term. The
classical Zsigmondy theorem~\cite{zsigmondy}, generalizing
earlier work of Bang~\cite{bang} (in the case~$b=1$), shows
that every term beyond the sixth in the
sequence~$(a^n-b^n)_{n\ge1}$ has a primitive divisor
(where~$a>b>0$ are coprime integers). Results of this form are
important in group theory and in the theory of recurrence
sequences (see the monograph~\cite[Sect.~6.3]{MR1990179} for a
discussion and references).

Our purpose here is to consider similar questions in polynomial
rings. The method of Carmichael~\cite{carmichael} is used to
find analogous results, with some modifications needed to avoid
terms in the sequence where the Frobenius automorphism
precludes primitive divisors. In even characteristic the
results take a slightly different form, and an analogue of
Bang's theorem is found here.

\section{Polynomial analogues}

Let~$k$ be a field (of odd characteristic, unless stated
otherwise), and consider a sequence~$(f_n)_{n\ge1}$ of elements
of~$k[T]$. Since~$k[T]$ is a unique factorization domain, each
term of the sequence factorizes into a product of irreducible
polynomials over~$k$, so we may ask which terms have an
irreducible factor which is not a factor of an earlier term.
Irreducible factors with this property will be called
\emph{primitive prime divisors}. As usual, we
write~$\ord_{\pi}f$ (or~$\ord_{p}n$) for
the maximal power to which an
irreducible~$\pi$ divides~$f$ in~$k[T]$ (or to which a rational
prime~$p$ divides~$n$ in~$\mathbb Z$).

The specific sequence we are interested in has~$f_n=f^n-g^n$,
where~$f,g$ are non-zero, coprime, polynomials in~$k[T]$.

\begin{lem}\label{lem:piadicord}
If~$\pi\in k[T]$ is an irreducible dividing~$f_n$ for
some~$n\ge1$, then for~$\cha(k)=p>0$,
\[
\ord_{\pi}(f_{mn})=p^{\ord_p(m)}\ord_{\pi}(f_n),
\]
and for~$\cha(k)=0$,
\[
\ord_{\pi}(f_{mn})=\ord_{\pi}(f_n).
\]
\end{lem}

\begin{proof}
We may write
\[
f^n-g^n=\pi^{\ord_{\pi}(f_n)}Q
\]
for some~$Q\in k[T]$ with~$\pi\notdivides Q$.
Write~$a=\ord_{\pi}(f_n)$, so
\[
f^{mn}
=
(g^n+\pi^{a}Q)^m
=
g^{mn}+\sum_{i=1}^m{m\choose i}
\pi^{ai}Q^ig^{n(m-i)}.
\]
Thus
\begin{equation}\label{eq:ordconclude}
f_{mn}=m\pi^{a}g^{n(m-1)}Q+\sum_{i=2}^m{m\choose i} \pi^{ai}Q^ig^{n(m-i)}.
\end{equation}
We deduce that if~$\cha(k)=p>0$, then for~$p\notdivides m$ (or
for~$\cha(k)=0$),
\[
\ord_{\pi}(f_{mn})=\ord_{\pi}(f_n).
\]
Now suppose that~$m=p^ek$ with~$e>0$ and~$p\notdivides k$. Then,
for~$\cha(k)=p>0$,
\[
f^{nm}-g^{nm}=(f^{nk}-g^{nk})^{p^e}.
\]
Now~$\ord_{\pi}(f_{nk})=\ord_{\pi}(f_n)$ since~$p\notdivides
k$, so~$\ord_{\pi}(f_{mn})=p^e\ord_{\pi}(f_n)$ as required.
\end{proof}

Recall that a sequence~$(f_n)$ is a divisibility sequence
if~$f_r\divides f_s$ whenever~$r\divides s$, and is a strong
divisibility sequence if~$\gcd(f_{r},f_{s})=f_{\gcd(r,s)}$ for all~$r,s\ge1$.

\begin{lem}\label{lem:sdlemma}
The sequence~$(f_n)_{n\geqslant 1}$ is a strong divisibility
sequence.
\end{lem}

\begin{proof}
Fix~$m,n\in\mathbb N$, and let~$\ell=\gcd(m,n)$. It is clear
that the sequence~$(f_n)$ is a divisibility sequence,
so~$f_{\ell}\divides\gcd(f_m,f_n)$. By B{\'e}zout's lemma there
exist~$c,d\in\mathbb{N}$ with~$\ell=cn-dm$, and
\begin{equation}\label{addedone}
f_{cn}(f^{dm}+g^{dm})-f_{dm}(f^{cn}+g^{cn})
=
2f^{dm}g^{dm}f_{\ell}.
\end{equation}
Any common divisor of~$f_m$ and~$f_n$ must divide~$f_{cn}$
and~$f_{dm}$. Since~$k$ has odd characteristic,~$2$ is a unit
in~$k[T]$ and so~\eqref{addedone} shows that any common divisor
of~$f_n$ and~$f_m$ divides~$f^{dm}g^{dm}f_{\ell}$. Since
both~$f$ and~$g$ are coprime to~$f_k$ for any~$k$, any divisor
of~$f_m$ and~$f_n$ divides~$f_{\ell}$, completing the proof.
\end{proof}

We will use the following simple observation several times.
Let~$K$ be a field, and let~$\Phi_d\in K[x,y]$ denote the~$d$th
homogeneous cyclotomic polynomial. If~$f,g\in K[T]$
have~$\deg(f)\neq\deg(g)$, then it is clear that~$\Phi_n(f,g)$
is not a unit for any~$n\in\mathbb{N}$. If~$\deg(f)=\deg(g)=d$,
and~$\zeta$ is a primitive~$n$th root of unity over~$K$, then
\[
\Phi_n(f,g)=\prod_{\genfrac{}{}{0pt}{}{i=1,}{\gcd(i,n)=1}}^{n}
(f-\zeta^ig).
\]
For~$\Phi_n(f,g)$ to be a unit requires that~$f-\zeta^ig$ is a
unit for each~$i$. Write
\[
f=\sum_{j=1}^da_jT^j, g=\sum_{j=1}^db_jT^j.
\]
For~$f-\zeta^ig$ to be a unit requires that~$a_d=\zeta^ib_d$.
Now for~$n>2$, the Euler function~$\phi(n)\geqslant 2$, and so we can
pick~$0<i_1<i_2<n$ with~$\gcd(i_1,n)=\gcd(i_2,n)=1$.
If~$a_d=\zeta^{i_1}b_d$ and~$a_d=\zeta^{i_2}b_d$, then
as~$a_d,b_d\neq 0$ by assumption, we must
have~$\zeta^{i_2-i_1}=1,$ contradicting the fact that~$\zeta$
is a primitive~$n$th root of unity. We deduce that, for coprime
polynomials~$f,g\in k[T]$,
\begin{equation}\label{addedeqnA}
\Phi_n(f,g)\mbox{ is not a unit if }n>2.
\end{equation}

These preparatory results give a polynomial form of Zsigmondy's
theorem as follows.

\begin{thm}\label{thm:fnfieldZsig}
Suppose~$\cha(k)=p>0$, and let~$P$ be the sequence obtained
from~$(f_n)_{n\geqslant 1}$ by deleting the terms~$f_n$
with~$p\divides n$. Then each term of~$P$ beyond the second has
a primitive prime divisor. If~$\cha(k)=0$, then the
sequence~$(f_n)_{n\geqslant 1}$ has the property that all terms
beyond the second have a primitive prime divisor.
\end{thm}

\begin{proof}
Notice that
\begin{equation}\label{eq:fnfield1}
f_n
=
\prod_{d\vert n}\Phi_d(f,g),
\end{equation}
and so
\[
\Phi_n(f,g)=\prod_{d\vert n}f_d^{\mu(n/d)}
\]
by M{\"o}bius inversion. Thus
\begin{equation}\label{eq:ordereqn}
\ord_{\pi}(\Phi_n(f,g))
=
\sum_{d\vert n}\mu(\textstyle\frac{n}{d})\ord_{\pi}(f_d)
\end{equation}
for any prime~$\pi\in k[T]$. Suppose now that~$\pi$ is a prime
divisor of~$f_n$ which is not primitive, so that~$\pi \divides
f_m$ for some~$m<n$ chosen to be minimal with that property.
Then~$m\divides n$ by Lemma~\ref{lem:sdlemma} and
\[
\ord_{\pi}(f_{mk})
=
\ord_{\pi}(f_{m})
\]
for any~$k$ with~$p\notdivides k$, by
Lemma~\ref{lem:piadicord}. In addition,
we claim it follows
that~$\ord_{\pi}(f_{c})=0$ unless~$m\divides c$. Suppose
this were not the case,
then~$\ord_{\pi}(f_{c})>0$ for some~$c$ with~$m\notdivides c$,
and Lemma~\ref{lem:sdlemma} yields $\pi\mid f_{\gcd(m,c)}$.
However, since~$m\notdivides c$,~$\gcd(m,c)<m$,
so this contradicts
the minimality of~$m$. Thus~\eqref{eq:ordereqn} gives
\begin{eqnarray}
\ord_{\pi}(\Phi_n(f,g))
&=&
\sum_{d\mid \frac{n}{m}}\mu(\textstyle\frac{n}{dm})\ord_{\pi}(f_{dm}) \nonumber \\
&=&
\sum_{d\mid \frac{n}{m}}\mu(\textstyle\frac{n}{dm})\ord_{\pi}(f_{m}) \nonumber \\
&=&
\ord_{\pi}(f_{m})\sum_{d\mid \frac{n}{m}}\mu(\textstyle\frac{n}{dm})=0\nonumber
\end{eqnarray}
as~$m<n$. We deduce that any non-primitive prime divisor
of~$f_n$ does not divide~$\Phi_n(f,g)$. By~\eqref{addedeqnA}
above,~$\Phi_n(f,g)$ is non-constant for~$n>2$, and
so~$\Phi_n(f,g)$ has a prime divisor in~$k[T]$. Therefore, as
any prime divisor of~$\Phi_n(f,g)$ is primitive, every term
in~$P$ beyond the second has a primitive prime divisor. The
proof for the characteristic zero case follows in exactly the
same way.
\end{proof}

We record two simple observations that arise from this
argument.
\begin{enumerate}
\item In fact~\eqref{eq:fnfield1} shows a little more: any
    primitive prime divisor of~$f_n$ must divide~$\Phi_n(f,g)$,
    and so the \emph{primitive part} (that is, the product of
    all the primitive prime divisors to their respective
powers) of~$f_n$ is
    exactly~$\Phi_n(f,g)$. This gives a lower bound
for the size of the primitive part~$f_n^*$ of~$f_n$
under
the assumption that~$\deg(f)\neq\deg(g)$:
\[
\deg(f_n^*)=\phi(n)\max\{\deg(f),\deg(g)\}>n^{1-\delta}
\max\{\deg(f),\deg(g)\}
\]
for~$\delta>0$ and large enough~$n$.
\item It is also clear that we need to remove all
    the terms from the sequence with index divisible by~$p$.
    If~$n=pc$ for some~$c\ge1$, then~$f_n=f_{pc}=(f_c)^p$,
    so any term with index divisible
    by~$p$ fails to have a primitive prime divisor.
\end{enumerate}

Theorem~\ref{thm:fnfieldZsig} is a form of Zsigmondy theorem
for polynomial rings, but it is not clear how to prove strong
divisibility when~$\cha(k)=2$. Computations suggest that the
result is still true in this case. When~$g=1$ and~$\cha(k)=2$,
the sequence~$(f_n)_{n\geqslant 1}$ satisfies the strong
divisibility property, giving the analogue of Bang's Theorem in
all characteristics.

\begin{lem}\label{lem:Banglem}
Let~$\cha(k)=2$ and let~$f\in k[T]$ be a non-zero non-unit.
Then the sequence $(h_n=f^n-1)_{n\geqslant1}$ is a strong
divisibility sequence.
\end{lem}

\begin{proof}
As before, let~$\ell=\gcd(m,n)$
so~$h_{\ell}\divides\gcd(h_m,h_n)$ by the divisibility
property. As before, there exist~$c,d\in\mathbb{N}$
with~$\ell=cn-dm$. A common divisor of~$h_n$ and~$h_m$ must
divide~$h_{cn}$ and~$h_{dm}$, and
\[
h_{cn}-h_{dm}=f^{dm}h_{\ell},
\]
so any common divisor of~$h_n$ and~$h_m$ must
divide~$f^{dm}h_{\ell}$. Since~$f$ and~$h_k$ are coprime for
any~$k$, any divisor of~$h_m$ and~$h_n$ must divide~$h_{\ell}$.
\end{proof}

\begin{cor}\label{cor:Bangfnfield}
Assume that~$\cha(k)=p\geqslant 2$ and~$h_n\in
k[T]$ is as in Lemma~\ref{lem:Banglem}.
Then the sequence obtained
from~$(h_n)_{n\ge1}$ by deleting terms with index
divisible by~$p$ has the property that all terms beyond the
first have a primitive prime divisor.
\end{cor}

\section{Polynomial Lucas sequences}

In this section we provide an analogue of the result of
Bilu, Hanrot and
Voutier
on primitive prime divisors in Lucas sequences.
Let~$k$ be a field, and fix~$\alpha\in\bar{k}$
such that~$[k(\alpha):k]=2$. Let~$\sigma$ be
the non-identity~$k$-automorphism of~$k(\alpha)$, and define the
polynomial sequence~$(L_n)_{n\ge1}$ by
\[
L_n=\frac{P^n-(P_{\sigma})^n}{P-P_{\sigma}},
\]
where for~$P=\sum_{i=0}^da_iT^i$ we
write~$P_{\sigma}=\sum_{i=0}^d\sigma(a_i)T^i$.
Then~$L_n\in k[T]$ and
we can again ask which terms
of the sequence see new irreducible factors.

We follow the path of Carmichael~\cite{carmichael} in deducing
some elementary arithmetic properties of the sequence.
In order to do this, there is a degenerate possibility that
must be avoided, so from now on we assume that~$P$ has
the property that~$P+P_{\sigma}$ and~$PP_{\sigma}$ are
coprime in~$k[T]$. Without this property, the sequence is not a strong
divisibility sequence. For example, if~$k=\mathbb{Q}$,~$\alpha=\sqrt{2}$,
and~$P=T^2+(1+\sqrt{2})T+\sqrt{2}$, then~$\gcd(L_2,L_3)=T+1\neq 1=L_1$.

\begin{lem}\label{lem:coprimelem}
The polynomials~$PP_{\sigma}$ and~$L_n$ are
coprime in~$k[T]$ for~$n\ge1$.
\end{lem}

\begin{proof}
The binomial expansion shows that
\begin{equation}\label{eq:diveq1}
(P+P_{\sigma})^{n-1}=P^{n-1}+(P_{\sigma})^{n-1}
+PP_{\sigma}Q_1 \end{equation} for some~$Q_1\in
k[T]$.
Moreover,
\begin{equation}\label{eq:diveq2}
L_n=P^{n-1}+(P_{\sigma})^{n-1}+PP_{\sigma}Q_2
\end{equation}
for some~$Q_2\in k[T]$.
If~$Q_3\in k[T]$ is irreducible and
divides both~$PP_{\sigma}$ and~$L_n$ then, by~\eqref{eq:diveq2},
we have~$Q_3\divides P^{n-1}+(P_{\sigma})^{n-1}$. Then,
by~\eqref{eq:diveq1},~$Q_3\divides P+P_{\sigma}$,
contradicting the standing assumption that~$P+P_{\sigma}$ and~$PP_{\sigma}$
are coprime. Thus the greatest common divisor of~$PP_{\sigma}$
and~$L_n$ must be a unit.
\end{proof}

As mentioned above, we deduce the strong divisibility
property for our sequence.

\begin{lem}\label{lem:sdlemma2}
Assume that~$\cha(k)\neq 2$. Then the sequence~$(L_n)_{n\geqslant
1}$ is a strong divisibility sequence.
\end{lem}

\begin{proof}
It is clear that~$(L_n)_{n\ge1}$ is a divisibility
sequence. As before, let~$\ell=\gcd(m,n)$ and
choose~$c,d\in\mathbb{N}$ with~$cn-dm=\ell.$ For brevity
write~$\widehat{L}_n=P^{n}+P_{\sigma}^{n}$, and notice that
\[
L_{cn}\widehat{L}_{dm}-L_{dm}\widehat{L}_{cn}=2(PP_{\sigma})^{dm}L_{\ell}.
\]
Hence a common divisor of~$L_n$ and~$L_m$
divides~$(PP_{\sigma})^{dm}L_{\ell}$, and hence must divide~$L_{\ell}$
by Lemma~\ref{lem:coprimelem}.
\end{proof}

The next result shows that in characteristic~$p$
we can still expect to find that, in general,
terms with index divisible by~$p$ once again fail
to produce primitive divisors.

\begin{lem}\label{lem:charp}
Let~$\cha(k)=p>2$. Then for~$n$ divisible by~$p$ (with
the possible exception of~$n=p$),~$L_n$ fails to have a primitive prime
divisor.
\end{lem}

\begin{proof}
Write~$L^{\prime}_n=P^{n}-P_{\sigma}^{n}$,
and assume that~$n=cp$ for some~$c\ge1$.
Then
\[
L^{\prime}_{cp}=(L^{\prime}_c)^p+
\sum_{i=1}^{(p-1)/2}(-1)^{i-1}{p\choose i}
(PP_{\sigma})^{ic}L^{\prime}_{(p-2i)c}.
\]
However~$p\divides{p\choose i}$ for~$1\leqslant i\leqslant \frac{p-1}{2}$,
so~$L^{\prime}_{cp}=(L^{\prime}_c)^p$,
and therefore~$L_{cp}=(L^{\prime}_1)^{p-1}L_c^p$.
\end{proof}

Thus, once again, terms whose index is divisible
by the characteristic must be removed in order to find
primitive divisors.

One more lemma is needed
before making the key
divisibility observation
for the sequences~$(L_n)_{n\ge1}$.

\begin{lem}\label{lem:keylem}
Assume that~$\cha(k)\neq2$.
Then~$\widehat{L}_m$ and~$L_m$ are coprime in~$k[T]$.
\end{lem}

\begin{proof}
Clearly
\[
\widehat{L}_m^2-(L^{\prime}_m)^2=4(PP_{\sigma})^m,
\]
so
\[
\widehat{L}_m^2-(L^{\prime}_1)^2L_m^2=4(PP_{\sigma})^m.
\]
By assumption,~$4$ is a unit in~$k[T]$, so any prime~$\pi\in k[T]$
dividing~$\widehat{L}_m$ and~$L_m$
also divides~$PP_{\sigma}$,
completing the proof by Lemma~\ref{lem:coprimelem}.
\end{proof}

\begin{lem}\label{lem:piadicord2}
Let~$L_n$ be as defined above. If~$\pi\in k[T]$
is a prime dividing~$L_n$, then for~$\cha(k)=p>0$
and~$m,n$ coprime to~$p$,
\[
\ord_{\pi}(L_{mn})=\ord_{\pi}(L_n),
\]
and for $\cha(k)=0$,
\[
\ord_{\pi}(L_{mn})=\ord_{\pi}(L_n).
\]
\end{lem}

\begin{proof}
For~$m$ odd, this proceeds as in the proof of Lemma~\ref{lem:charp}.
The result is clearly true for~$m=1$.
So now suppose that
\[
\ord_{\pi}(L_{bn})=\ord_{\pi}(L_n)
\]
for each odd integer~$b<m$.
Then we note that
\[
L_{mn}=(P-P_\sigma)^{m-1}L_n^m+\sum_{i=1}^{{(m-1)}/{2}}(-1)^i{m\choose i}
(PP_{\sigma})^{in}L_{(m-2i)n}.
\]
Not all terms inside the summation are zero, since~$m$ is
coprime to~$p$,
so by the inductive assumption we conclude the
statement of the lemma
by the ultrametric property of the
valuation~$\ord_{\pi}$.
For~$m$ even, note that it is sufficient to prove this
for~$m=2$. However, since
\[
L_{2m}=\frac{P^{2m}-P_{\sigma}^{2m}}{P-P_{\sigma}}
=\frac{P^{m}-P_{\sigma}^m}{P-P_{\sigma}}\cdot (P^m+P_{\sigma}^m),
\]
we see that
\[
L_{2m}=\widehat{L}_mL_m.
\]
By Lemma~\ref{lem:keylem},~$\hat{L}_m,L_m$ are
coprime in~$k[T]$, and so
\[
\ord_{\pi}(L_{2m})=\ord_{\pi}(L_m).
\]
\end{proof}

As before, we are now ready for our Zsigmondy theorem.

\begin{thm}\label{thm:fnfieldZsig3}
Suppose~$\cha(k)=p>2$, and let~$Q$ be the sequence obtained
from~$(L_n)_{n\geqslant 1}$ by deleting the terms
with~$p\divides n$. Then each term of~$Q$ beyond the second has a
primitive prime divisor. If~$\cha(k)=0$, then the
sequence~$(L_n)_{n\geqslant 1}$ has the property that all terms beyond
the second have a primitive prime divisor.
\end{thm}

\begin{proof}
We begin by noting the fact that
\begin{equation}
L_n=
\prod_{\genfrac{}{}{0pt}{}{d\vert n,}{d>1}}
\Phi_d(P,P_{\sigma}),\nonumber
\end{equation}
where~$\Phi_d$ is the~$d$th homogeneous cyclotomic polynomial.
By M\"{o}bius inversion,
\[
\Phi_n(P,P_{\sigma})=\prod_{\genfrac{}{}{0pt}{}{d\vert n,}{d>1}}L_d^{\mu(n/d)}
=\prod_{d\vert n}L_d^{\mu(n/d)}.
\]
The rest of the proof proceeds along the same lines as the
proof of Theorem~\ref{thm:fnfieldZsig}, combining
Lemmas~\ref{lem:sdlemma2} and~\ref{lem:piadicord2}
with~\eqref{addedeqnA}.

\end{proof}


\providecommand{\bysame}{\leavevmode\hbox to3em{\hrulefill}\thinspace}

\end{document}